\documentclass{amsart}

\usepackage[colorlinks=true, urlcolor=blue, citecolor=blue, linkcolor=blue, hypertexnames=false]{hyperref}
\usepackage{amssymb}
\usepackage{amsfonts}
\usepackage{amsmath}
\usepackage{amsthm}
\usepackage{mathrsfs}
\usepackage{dsfont}
\usepackage{mathtools}

\usepackage{graphicx}
\newcommand{\rotmodels}{\mathbin{\rotatebox[origin=c]{90}{$\models$}}}
\newcommand{\Bot}{\!\rotmodels\!}
\newcommand{\srotmodels}{\mathbin{\rotatebox[origin=c]{90}{\scalebox{.7}{$\models$}}}}
\newcommand{\sBot}{\srotmodels}

\usepackage{aliascnt}

\newtheorem{thm}{Theorem}[section]

\newaliascnt{lem}{thm}
\newtheorem{lem}[lem]{Lemma}
\aliascntresetthe{lem}

\newaliascnt{prp}{thm}  
\newtheorem{prp}[prp]{Proposition}
\aliascntresetthe{prp}

\newaliascnt{cor}{thm}  
\newtheorem{cor}[cor]{Corollary}
\aliascntresetthe{cor}

\theoremstyle{definition}

\newaliascnt{dfn}{thm}  
\newtheorem{dfn}[dfn]{Definition}
\aliascntresetthe{dfn}

\numberwithin{equation}{section}

\author{Tristan Bice}
\address{
Federal University of Santa Catarina\\
Florianopolis\\
Brazil
}
\email{Tristan.Bice@gmail.com}
\thanks{This research has been supported by a CAPES (Brazil) postdoctoral fellowship through the program ``Science without borders'', PVE project 085/2012.}
\keywords{Involution Semigroup, Annihilator, Hereditary Subalgebra, Projection, Ortholattice, C*-Algebra, Von Neumann Algebra, Murray-von Neumann Equivalence}
\subjclass[2010]{Primary: 20M10, 20M12, 46L85; Secondary: 46L05, 46L10, 46L45}

\begin{document}

\title{*-Annihilators in Proper *-Semigroups}

\begin{abstract}
We generalize some basic C*-algebra and von Neumann algebra theory on hereditary C*-subalgebras and projections.  In particular, we extend Murray-von Neumann equivalence from projections to *-annihilators and show that several of its important properties also extend from von Neumann algebras to proper *-semigroups.  Furthermore, we show how to obtain *-annihilator decompositions of a proper *-semigroup that correspond to classical von Neumann algebra type decompositions.
\end{abstract}

\maketitle

\section{Introduction}

\subsection{Motivation}

*-Annihilators underpin a lot of basic von Neumann algebra theory, although this is far from clear in most of the literature on the subject.  This is because *-annihilators correspond to projections and it is invariably the projections that people actually work with.  It is even tempting to think that this correspondence between *-annihilators and projections is crucial to the theory, and this is why various generalizations of von Neumann algebras, like AW*-algebras, Baer *-rings (see \cite{Berberian1972}) and complete Baer *-semigroups (see \cite{Foulis1960} or \cite{Harding2013}), take this correspondence as a fundamental assumption from the outset.

The purpose of the present paper is to show that this correspondence is not as vital as it seems, and that by working with *-annihilators directly we can still generalize much of the basic theory.  The only key assumption we need throughout is properness (see \autoref{propsemi}), which holds in a much broader class of algebraic structures, like the C*-algebras which originally motivated our work (see \cite{Bice2014}).  Indeed, a number of results about projections in von Neumann algebras have so far resisted generalization to C*-algebras, and the present paper would suggest the reason for this is simple \textendash\, they were really *-annihilator results in disguise.

\subsection{Outline}

We start off in \S\ref{NCT} by considering slightly more general subsets of a *-semigroup which correspond to open subsets of a topological space, just as *-annihilators correspond to regular open subsets.  Next we examine some basic properties of proper *-semigroups in \S\ref{P*} and then move on to the orthogonality relations that naturally give rise to *-annihilators in \S\ref{OR}.  We begin our investigation of the *-annihilator ortholattice $\mathscr{P}(S)^\perp$ in earnest in \S\ref{*AO}, showing that the *-annihilator ideals $\mathscr{P}(S)^\triangledown$ take the place of central projections in a von Neumann algebra, even yielding type decompositions in \autoref{typeI} and \autoref{typeI1} which correspond to parts of the classic von Neumann algebra type decomposition.

In \S\ref{eq} we introduce the $\sim$ relation on *-annihilators that naturally extends Murray-von Neumann equivalence of projections (see \autoref{MvN}).  Moreover we show that, as long as $\sim$ satisfies certain reflexivity and additivity assumptions, then $\sim$ retains many of its important properties even in the much more general context of proper *-semigroups.  For example, we show $\sim$ is weaker than perspectivity (see \autoref{simper}), satisfies an analog of the Cantor-Schroeder-Bernstein Theorem (see \autoref{CSB}), satisfies generalized comparability (see \autoref{gencom}), and also has properties sufficient to yield type decompositions (see \autoref{typeIII} and \autoref{typefin}) corresponding to the classic finite/properly infinite  and semifinite/purely infinite decompositions of von Neumann algebras.  Lastly, we point out in \S\ref{*R} that the reflexivity and additivity assumptions on which these results rely hold in certain proper *-rings, including all separable C*-algebras.  Indeed, the next natural step would be to move on to the C*-algebra case, where Hilbert space representations and the continuous functional calculus allow even more to be said about the *-annihilator ortholattice, as we demonstrate in a forthcoming paper based on \cite{Bice2014}.  

\section{Non-Commutative Topology}\label{NCT}

Say we are given a locally compact topological space $X$ which we want to analyze in a more algebraic way.  To do this we associate to $X$ a certain algebraic structure $A$ so that open subsets of $X$ correspond to certain subsets of $A$.  The standard way to do this is to let $A$ be the commutative C*-algebra $C_0(X,\mathbb{C})$ and then the open subsets of $X$ correspond to ideals of $A$ that are closed (w.r.t. uniform convergence).  In general possibly non-commutative C*-algebras, we can replace closed ideals with norm closed left (or right) ideals, hereditary *-algebras or hereditary cones.  While there is no concrete topological space $X$ associated to a non-commutative C*-algebra, we can still show that these non-commutative analogs of open subsets are in a natural bijective correspondence to each other (see \cite{Pedersen1979} Theorem 1.5.2).

We could also let $A$ be the unit ball of the above C*-algebra, more explicitly $C_0(X,\mathbb{C}^1)$, where $\mathbb{C}^1=\{z\in\mathbb{C}:|z|\leq1\}$.  Then $A$ is no longer a C*-algebra, as it is not closed under addition or scalar multiplication, however it is still a commutative *-semigroup.  Moreover, the open subsets of $X$ now correspond to the closed ideals of $A$ whose positive parts are closed under square roots and finite supremums.  The purpose of the present section is to show that in arbitrary *-semigroups we again get natural correspondences between various non-commutative analogs of open subsets, even without closure in norm or supremums.  To precisely describe these analogs and the correspondences between them we first need some definitions from \cite{CliffordPreston1961}.

\begin{dfn}
A \emph{semigroup} is a set $S$ together with a single associative binary operation $\cdot$ (usually abbreviated to juxtaposition).  We call $I\subseteq S$
\begin{itemize}
\item a \emph{subsemigroup} if $II\subseteq I$.
\item a \emph{left (right) ideal} if $SI\subseteq I$ ($IS\subseteq I$).
\item an \emph{ideal}, if $I$ is both a left and right ideal.
\item a \emph{quasi-ideal}, if $IS\cap SI\subseteq I$ and $I$ is a subsemigroup of $S$.
\item a \emph{bi-ideal}, if $ISI\subseteq I$ and $I$ is a subsemigroup of $S$.
\end{itemize}
\end{dfn}

If $S$ has a ring structure then we could also talk about ring ideals, although these are more restrictive as they must also be closed under addition (although *-annihilators in a *-ring are always closed under addition, see \autoref{Lbotprp}\eqref{Lbotprp6}).  Indeed, a union of semigroup ideals is again a semigroup ideal, while this is rarely the case for ring ideals.  For us, ideal always means semigroup ideal.

\begin{dfn}
A *-semigroup $S$ is a semigroup together with a function $^*:S\rightarrow S$ that is an involutive ($s^{**}=s$) antihomomorphism ($(st)^*=t^*s^*$).  Given $T\subseteq S$, we define $\sqrt{T}=\{s:s^*s\in T\}$, $T^2=\{t^*t:t\in T\}$, $T_+=T\cap S^2$ and call $T$
\begin{itemize}
\item \emph{self-adjoint} if $T=T^*$.
\item a \emph{*-subsemigroup} if $T$ is a self-adjoint subsemigroup.
\item \emph{left (right) rooted} if $\sqrt{T}\subseteq T$ ($\sqrt{T}^*\subseteq T$), for all $t\in T$.
\item \emph{rooted} if $T$ is both left and right rooted.
\item \emph{quasi-rooted} if $\sqrt{T}\cap\sqrt{T}^*\subseteq T$.
\item \emph{positive rooted} if $\sqrt{T}_+\subseteq T$
\item \emph{hereditary} if $t^*St\subseteq T$ whenever $t^*t\in T$.
\item \emph{positive hereditary} if $t^*S_+t\subseteq T$ whenever $t^*t\in T$.
\item \emph{bi-hereditary}, if $tS_+t\subseteq T$, whenever $t\in T_+$.
\end{itemize}
\end{dfn}

Note that, as $t^*tS_+t^*t\subseteq t^*S_+t$, any positive hereditary $T$ is bi-hereditary and clearly any bi-ideal is also bi-hereditary.  Also note that, for any $s=t^*t\in S_+$ and $n\in\mathbb{N}$, we have
\begin{equation}\label{s^n}
s^n\in S_+,
\end{equation}
for if $n$ is even then $s^n=s^{n/2*}s^{n/2}$, while if $n$ is odd then $s^n=(ts^{(n-1)/2})^*(ts^{(n-1)/2})$.

We first observe some basic properties of the maps $I\mapsto\sqrt{I}$ and $I\mapsto I_+$.

\begin{prp}\label{sqrtI}
If $S$ is a *-semigroup and $I\subseteq S$ then $\sqrt{I}$ is
\begin{itemize}
\item a left ideal, if $I$ is positive hereditary, and
\item left rooted, if $I$ is positive rooted.
\end{itemize}
\end{prp}

\begin{proof}
If $t\in\sqrt{I}$ then $t^*t\in I$ so if $I$ is positive hereditary then $t^*s^*st\in I$, for any $s\in S$, which means $st\in\sqrt{I}$, i.e. $\sqrt{I}$ is a left ideal.  If $s^*s\in\sqrt{I}$ then $s^*ss^*s\in I$ which, if $I$ is positive rooted, means $s^*s\in I$ and hence $s\in\sqrt{I}$, i.e. $\sqrt{I}$ is left rooted.
\end{proof}

\begin{prp}
If $S$ is a *-semigroup and $I\subseteq S$ then $I_+$ is
\begin{itemize}
\item positive rooted, if $I$ is left rooted, and
\item positive hereditary and equal to $I^2$, if $I$ is a left rooted left ideal.
\end{itemize}
\end{prp}

\begin{proof}
If $I$ is left rooted then $(\sqrt{I_+})_+=\sqrt{I}_+\subseteq I_+$, i.e. $I_+$ is positive rooted.  While if $t^*t\in I$ then $t\in I$ and $t^*t\in I^2$, as long as $I$ is left rooted, and hence $t^*S_+t\subseteq I_+$, as long as $I$ is also a left ideal, i.e. $I_+\subseteq I^2$ is positive hereditary.  Furthermore, $s\in I^2$ means $s=t^*t$ for some $t\in I$ and hence $s=t^*t\in I$, as long as $I$ is again a left ideal, so $I^2\subseteq I_+$.
\end{proof}

We can now prove these maps yield one of the desired correspondences.

\begin{thm}\label{correspondence1}
If $S$ is a *-semigroup then the maps $I\mapsto I_+=I^2$ and $I\mapsto\sqrt{I}$ are mutually inverse order isomorphisms between left rooted left ideals and positive rooted positive hereditary subsets of $S_+$.
\end{thm}

\begin{proof}
If $s\in I$ and $I$ is a left ideal then $s^*s\in I$ and hence $s\in\sqrt{I}=\sqrt{I_+}$, i.e. $I\subseteq\sqrt{I_+}$, while if $I$ is also left rooted then $\sqrt{I_+}=\sqrt{I}\subseteq I$.  Thus $I=\sqrt{I_+}$, for any left rooted left ideal $I$.  On the other hand if $s=t^*t\in I\subseteq S_+$ and $I$ is positive hereditary then $s^2=t^*(tt^*)t\in I$ so $s\in\sqrt{I}_+$, i.e. $I\subseteq\sqrt{I}_+$, while if $I$ is positive rooted then $\sqrt{I}_+\subseteq I$.  Thus $I=\sqrt{I}_+$, for any positive rooted positive hereditary $I$.  Now the result follows from the previous two propositions.
\end{proof}

Next we note some basic properties of the map $I\mapsto I\cap I^*$.

\begin{prp}\label{IcapI*}
If $S$ is a *-semigroup and $I\subseteq S$ then $I\cap I^*$ is self-adjoint and
\begin{itemize}
\item a quasi-ideal, if $I$ is a left ideal,
\item quasi-rooted, if $I$ is left rooted, and
\item hereditary, if $I$ is a left rooted left ideal.
\end{itemize}
\end{prp}

\begin{proof}
As $^*$ is involutive, $I\cap I^*$ is self-adjoint.  As $^*$ is an antihomomorphism, if $I$ is a subsemigroup of $S$ then so is $I^*$ and hence $I\cap I^*$.  In particular this applies if $I$ is a left ideal and then $I^*$ is a right ideal so $S(I\cap I^*)\cap(I\cap I^*)S\subseteq SI\cap I^*S\subseteq I\cap I^*$, i.e. $I\cap I^*$ is a quasi-ideal.  If $I$ is left rooted and $ss^*,s^*s\in I\cap I^*\subseteq I$ then $s\in I\cap I^*$, i.e. $I\cap I^*$ is quasi-rooted.  If $s\in S$ and $t^*t\in I\cap I^*$ then $t\in I$, if $I$ is left rooted, so $t^*st,t^*s^*t\in I$, if $I$ is also a left ideal, so $t^*st\in I\cap I^*$, i.e. $I\cap I^*$ is hereditary.
\end{proof}

We can now prove the other desired correspondence.

\begin{thm}\label{thecorrespondence}
If $S$ is a *-semigroup, the maps $I\mapsto I^*\cap I$ and $I\mapsto\sqrt{I}$ are isomorphisms between left rooted left ideals and quasi-rooted hereditary *-subsemigroups.
\end{thm}

\begin{proof}
As in the proof of \autoref{correspondence1} $I=\sqrt{I_+}=\sqrt{I\cap I^*}$, for any left rooted left ideal $I$.  On the other hand if $s\in I$ and $I$ is a *-subsemigroup of $S$ then $s^*\in I$ and $ss^*,s^*s\in I$ so $s\in\sqrt{I}\cap\sqrt{I}^*$, i.e. $I\subseteq\sqrt{I}\cap\sqrt{I}^*$, while if $I$ is quasi-rooted then $\sqrt{I}\cap\sqrt{I}^*\subseteq I$.  Now the result follows from \autoref{sqrtI} and \autoref{IcapI*}.
\end{proof}

\begin{cor}
Any quasi-rooted hereditary *-subsemigroup is a quasi-ideal.
\end{cor}

\begin{proof}
If $I$ is a quasi-rooted hereditary *-subsemigroup then, by \autoref{thecorrespondence}, $\sqrt{I}$ is a left ideal and $I=\sqrt{I}\cap\sqrt{I}^*$, which is a quasi-ideal by \autoref{IcapI*}.
\end{proof}

One useful thing to note about left rooted left ideals and quasi-rooted *-subsemigroups is that inclusion is completely determined by positive elements.

\begin{prp}\label{pospart}
If $I,J\subseteq S$ are left rooted left ideals, or quasi-rooted *-subsemigorups, \[I\subseteq J\quad\Leftrightarrow\quad I_+\subseteq J_+.\]
\end{prp}

\begin{proof}
The forward implication is immediate in either case.  Conversely, say we have $I,J\subseteq S$ and $s\in I\setminus J$.  If $I$ is a left ideal then $s^*s\in I_+$, and if $J$ is left rooted then $s^*s\in J_+$ would imply $s\in J$, a contradiction.  While if $I$ is a *-subsemigroup then $s^*\in I$ and $s^*s,ss^*\in I$, and if $J$ is quasi-rooted then $s^*s,ss^*\in J_+$ would imply that $s\in J$, a contradiction.  Thus in either case $I\nsubseteq J$ implies $I_+\nsubseteq J_+$ (alternatively, in the left rooted left ideal case, simply note that $I=\sqrt{I_+}$ and $J=\sqrt{J_+}$, by \autoref{correspondence1}).
\end{proof}

\section{Proper *-Semigroups}\label{P*}

While the subsets of a *-semigroup examined in the previous section are natural non-commutative analogs of open subsets, our primary interest actually lies in the analogs of regular open subsets (those open $O$ that are the interior of their closure $\overline{O}^\circ$).  To define these we need some $0$ in our *-semigroup with special properties which we now describe and investigate.

\begin{dfn}
If $S$ is a semigroup then we call $s\in S$
\begin{itemize}
\item \emph{idempotent} if $\{s\}$ is a subsemigroup.
\item \emph{(left/right/quasi-/bi-) absorbing} if $\{s\}$ is a (left/right/quasi-/bi-) ideal.
\end{itemize}
\end{dfn}

Clearly any left or right ideal is a subsemigroup, so in particular any absorbing element is idempotent.  And if $0_\mathrm{l}$ is left absorbing and $0_\mathrm{r}$ is right absorbing then $0_\mathrm{l}=0_\mathrm{r}0_\mathrm{l}=0_\mathrm{r}$, i.e. then there is unique (left/right) absorbing element.  Likewise, different quasi-absorbing elements can not commute.

\begin{dfn}
If $S$ is a *-semigroup, we call $p\in S$
\begin{itemize}
\item \emph{positive} if $p\in S_+$.
\item \emph{self-adjoint} if $\{p\}$ is self-adjoint.
\item a \emph{projection} if $\{p\}$ is a *-subsemigroup or, equivalently, if $\{p\}^2=\{p\}$.
\item \emph{(left/right/quasi-) proper} if $\{p\}$ is (left/right/quasi-) rooted.
\end{itemize}
\end{dfn}

Denoting the idempotent, self-adjoint, and projection elements of $S$ by $S_\mathrm{idem}$, $S_\mathrm{sa}$, and $S_\mathrm{proj}$ respectively, we have $S_+\subseteq S_\mathrm{sa}$ and $S_\mathrm{proj}=S_\mathrm{sa}\cap S_\mathrm{idem}=S_+\cap S_\mathrm{idem}$.

\begin{prp}
If $S$ is a *-semigroup and $0\in S$, the following are equivalent.
\begin{itemize}
\item $0$ is proper and absorbing.
\item $0$ is left proper and left absorbing.
\item $0$ is left proper, bi-absorbing and self-adjoint.
\end{itemize}
\end{prp}

\begin{proof}
If $0$ is left absorbing then $0=0^*0\in A_+\subseteq A_\mathrm{sa}$.  If $0$ is bi-absorbing and self-adjoint then, for any $s\in S$, $0=0s^*s0=0^*s^*s0=(s0)^*s0$ which, if $0$ is also left proper, means $s0=0$, i.e. $0$ is left absorbing.  As $0$ is self-adjoint, it must then also be right proper and right absorbing.
\end{proof}

\autoref{propsemi} and \autoref{s*st} below are straightforward generalizations of \cite{Berberian1972} \S2 Definition 1 and Proposition 1 from *-rings to *-semigroups.

\begin{dfn}\label{propsemi}
A *-semigroup $S$ is \emph{proper} if there is a proper absorbing $0\in S$.
\end{dfn}

A *-semigroup $S$ with absorbing $0\in S$ is proper, as defined here, if it satisfies *-cancellation (see \cite{Foulis1963}), namely the property that $aa^*=ab^*=bb^*\Rightarrow a=b$, for all $a,b\in S$, a concept which some authors take as the definition of properness (see \cite{Drazin1978}).  In any case, properness is equivalent to *-cancellation in *-rings.

\begin{prp}\label{s*st}
If $S$ is a proper *-semigroup and $s,t\in S$ then \[st=0\quad\Leftrightarrow\quad s^*st=0.\]
\end{prp}

\begin{proof}
If $st=0$ then $s^*st=0$, as $0$ is absorbing.  While if $s^*st=0$ then $(st)^*st=t^*s^*st=0$, as $0$ is absorbing, and hence $st=0$, as $0$ is proper.
\end{proof}

\begin{cor}\label{pq}
If $S$ is a proper *-semigroup, $s\in S$ and $p,q\in S_+$ then
\begin{eqnarray*}
p^2s=0 &\Leftrightarrow& ps=0,\\
pqps=0 &\Leftrightarrow& qps=0,\textrm{ and}\\
pqp=0 &\Leftrightarrow& qp=0
\end{eqnarray*}
\end{cor}

\begin{proof}
If $ps=0$ then, as $0$ is absorbing, $p^2s=0$.  Conversely, if $0=p^2s=p^*ps$ then $ps=0$, by \autoref{s*st}.  Likewise, $qps=0\Rightarrow pqps=0$ while $0=pqps=(rp)^*rps\Rightarrow0=rps\Rightarrow 0=r^*rps=qps$, where $r^*r=q$.  The last equation follows similarly (even without recourse to \autoref{s*st}, as we can use properness directly).
\end{proof}

By iterating \autoref{pq}, we see that any arbitrarily long multiple of $p,q\in S_+$ will be $0$ if and only if $pq=0$.  However, we can not quite do the same for triples $p,q,s\in S_+$, specifically $qp^ns=0$ does not necessarily imply that $qps=0$.  For example, if $S=M_2(\mathbb{Z})$, the (proper multiplicative *-semigroup of) $2\times2$ $\mathbb{Z}$-matrices, $q=\begin{bmatrix} 1 & 1 \\ 1 & 1 \end{bmatrix}^2$, $p=\begin{bmatrix} 1 & 0 \\ 0 & 2 \end{bmatrix}^2$ and $s=\begin{bmatrix} 16 & -4 \\ -4 & 1 \end{bmatrix}^2$ then $qp^ns=0$ iff $n=1$.

\section{Orthogonality Relations}\label{OR}

Our non-commutative analogs of regular open subsets are best derived as the closed subsets with respect to certain polarities, which in turn come from the following relations that we now define and examine.

\begin{dfn}
On a semigroup $S$ with $0\in S$ we define a binary relation $\triangledown$ by
\[s\triangledown t\quad\Leftrightarrow\quad \{0\}=sSt.\]
If $S$ is a *-semigroup we also define binary relations $\mathrm{L}$ and $\perp$ by
\begin{eqnarray*}
s\mathrm{L} t &\Leftrightarrow& 0=st^*\\
s\bot t &\Leftrightarrow& 0=st^*=st\\
\end{eqnarray*}
\end{dfn}

The basic relationships between these relations are as follows.

\begin{prp}\label{triequiv}
If $S$ is a *-semigroup, $0,s,t,u\in S$ and $0=0^*$ then \[s\triangledown t\ \Leftrightarrow\ t^*\triangledown s^*,\quad s\bot t\ \Leftrightarrow\ s\bot t^*\ \Rightarrow\ s\mathrm{L}t\ \Leftrightarrow\ t\mathrm{L}s,\quad\textrm{and}\quad s\mathrm{L}tu\ \Leftrightarrow\ su^*\mathrm{L}t.\]
If $S$ is also proper, i.e. if $0$ is proper and absorbing, then \[s^*s\triangledown t\quad\Leftrightarrow\quad s^*\triangledown t\quad\Leftrightarrow\quad s\triangledown t\quad\Leftrightarrow\quad t\triangledown s\quad\Rightarrow\quad s\bot t.\]
If $S$ is also commutative then \[s\triangledown t\quad\Leftrightarrow\quad s\bot t\quad\Leftrightarrow\quad s\mathrm{L}t\quad\Leftrightarrow\quad st=0.\]
\end{prp}

\begin{proof}
The first statement is immediate.  For the second statement, assume $S$ is proper and $s^*St=\{0\}$.  Then $s^*sSt=\{0\}$ and hence $sSt=\{0\}$, by \autoref{s*st}, thus proving the first two equivalences.  The third equivalence follows from this and $s\triangledown t\Leftrightarrow t^*\triangledown s^*$.  For the last implication, note that if $t\triangledown s$ then $ts^*s=0$ and hence $ts^*=0=st^*$, by \autoref{s*st}, and also $tt^*s^*s=0$ and hence $t^*s^*=0=st$, again by \autoref{s*st}, i.e. $s\bot t$.  Finally, for the third statement, if $S$ is also commutative then $s\mathrm{L}t$ means $st^*=0=st^*t=stt^*$ so $st=0$, by \autoref{s*st}, and hence $sSt=stS=\{0\}$, i.e. $s\triangledown t$.
\end{proof}

In any commutative proper *-semigroup $S$ we have $ss^*=0\Leftrightarrow ss=0$, by the above result, and hence $S$ will contain no non-zero nilpotents.  While any commutative semigroup $S$ with no non-zero nilpotents becomes a proper *-semigroup when we define $s^*=s$, for all $s\in S$.  Thus any result proved about the relations $\mathrm{L}$, $\perp$ and $\triangledown$ in a proper *-semigroup is, in the commutative case,  really just a result about the relation $st=0$ in a commutative semigroup with no non-zero nilpotents.

If $I$ is a *-subsemigroup of a *-semigroup $S$ with common absorbing $0$, we write $\triangledown_I$ and $\bot_I$ for the corresponding relations defined within $I$.  We immediately see that $\bot_I$ is simply the restriction of $\bot$ to $I$ and with suitable assumptions the same is true of $\triangledown_I$.

\begin{prp}\label{triT}
If $I$ is a *-subsemigroup and bi-ideal of a proper *-semigroup $S$ and $s,t\in I$ then $s\triangledown_It$ if and only if $s\triangledown t$.
\end{prp}

\begin{proof}
If $sSt\neq0$ then $sut\neq0$, for some $u\in S$, so $ss^*sutt^*t\neq0$, by \autoref{s*st}, and hence $sIt\neq0$, as $I$ is a bi-ideal.  Thus $s\triangledown t$ if (and only if) $s\triangledown_It$.
\end{proof}

Note that any binary relation $R$ gives rise to a polarity, i.e. an (inclusion) order reversing function $^R:\mathscr{P}(S)\rightarrow\mathscr{P}(S)$ defined by $T^R=\{s\in S:\forall t\in T(tRs)\}$.

\begin{prp}
If $S$ is a semigroup, $0$ is absorbing and $T\subseteq S$, $T^\triangledown$ is an ideal.
\end{prp}

\begin{proof}
As $TSST^\triangledown\subseteq TST^\triangledown=\{0\}$, we have $ST^\triangledown\subseteq T^\triangledown$.  As $0$ is (right) absorbing, we have $TST^\triangledown S=\{0\}S=\{0\}$, i.e. $T^\triangledown S\subseteq T^\triangledown$ too.
\end{proof}

In a proper *-semigroup we also have the the following properties of $^\triangledown$, $^\bot$ and $^\mathrm{L}$.

\begin{lem}\label{Lbotprp}
If $S$ is a proper *-semigroup and $T\subseteq S$ then
\begin{enumerate}
\item\label{Lbotprp1} $T^\triangledown$ is rooted ideal (and hence self-adjoint).
\item\label{Lbotprp2} $T^\mathrm{L}=\sqrt{T^\bot}=T^{2\mathrm{L}}$ is a left rooted left ideal.
\item\label{Lbotprp3} $T^\bot=T^\mathrm{L}\cap T^{\mathrm{L}*}$ is a (self-adjoint) quasi-rooted hereditary quasi-ideal.
\item\label{Lbotprp4} $T^{\mathrm{L}\mathrm{L}}=T^{\bot\mathrm{L}}$.
\item\label{Lbotprp5} $T^\triangledown=T^\bot=T^\mathrm{L}$, if $T$ is a right ideal.
\item\label{Lbotprp6} $T^\triangledown$, $T^\bot$ and $T^\mathrm{L}$ are subrings, if $S$ is a *-ring.
\item\label{Lbotprp7} $T^\triangledown$, $T^\bot$ and $T^\mathrm{L}$ are closed, if $s\mapsto ts$ and $s\mapsto s^*$ are continuous.
\end{enumerate}
\end{lem}

\begin{proof}
By \autoref{triequiv}, $T^\triangledown$ is rooted.  If $st^*=0$ then $s(ut)^*=st^*u^*=0$, so $T^\mathrm{L}$ is a left ideal.  As $0=s(t^*t)^*=st^*t\Leftrightarrow st^*=0$, $T^\mathrm{L}=\sqrt{T^\bot}$, $T^\mathrm{L}$ is left rooted and $T^{\mathrm{L}2}=(T^\mathrm{L})_+$.  We also immediately have $T^\bot=T^\mathrm{L}\cap T^{\mathrm{L}*}$, so $T^\bot$ is a quasi-rooted hereditary quasi-ideal, by \autoref{IcapI*} (this can also be verified directly).  As $s^*st=0\Leftrightarrow st^*=0$, we have $T^\mathrm{L}=T^{2\mathrm{L}}$.  As $(T^\mathrm{L})_+\subseteq T^\mathrm{L}\cap T^{\mathrm{L}*}=T^\bot\subseteq T^\mathrm{L}$, we have \[T^{\mathrm{L}\mathrm{L}}\subseteq T^{\bot\mathrm{L}}\subseteq((T^\mathrm{L})_+)^\mathrm{L}=T^{\mathrm{L}2\mathrm{L}}=T^{\mathrm{L}\mathrm{L}}.\]
If $T$ is a right ideal and $t\in T^\mathrm{L}$ then $TSt^*\subseteq Tt^*=\{0\}$, i.e. $t^*\in T^\triangledown=T^{\triangledown*}$, so $T^\mathrm{L}\subseteq T^\triangledown\subseteq T^\bot\subseteq T^\mathrm{L}$.  If $S$ is a *-ring and $st^*=0=su^*$ then $s(t+u)^*=0$ so $T^\mathrm{L}$ and likewise $T^\triangledown$ and $T^\bot$ are closed under addition.  If $S$ has a topology in which the given maps are continuous, $t_\alpha\rightarrow t$ and $st_\alpha^*=0$, for all $\alpha$, then $st^*=0$ so $T^\mathrm{L}$ and, likewise, $T^\triangledown$ and $T^\bot$ are topologically closed.
\end{proof}

In the notation below, $\mathscr{P}(S)^\mathrm{R}=\{T^\mathrm{R}:T\subseteq S\}$.

\begin{thm}\label{Lperpequiv}
If $S$ is a proper *-semigroup then $\mathscr{P}(S)^\triangledown$, $\mathscr{P}(S)^\bot$ and $\mathscr{P}(S)^\mathrm{L}$ are complete ortholattices in which $\bigwedge\mathcal{T}=\bigcap\mathcal{T}$.  Moreover, $\mathscr{P}(S)^\mathrm{L}\cong\mathscr{P}(S)^\bot$ and
\[\mathscr{P}(S)^\triangledown=\mathscr{P}(S)^\mathrm{L}\cap\mathscr{P}(S)^\bot=\{T\in\mathscr{P}(S)^\mathrm{L}:T=T^*\}=\{T\in\mathscr{P}(S)^\bot:ST\subseteq T\}.\]
\end{thm}

\begin{proof}
As $\mathrm{L}$ is symmetric and $s\bot s\Leftrightarrow s\in S^\mathrm{L}=\{0\}$, by properness, it follows immediately (see \cite{Birkhoff1967} Ch V \S7) that $\mathscr{P}(S)^\mathrm{L}$ is a complete ortholattice in which infimums are intersections, and likewise for $\triangledown$.  By \autoref{thecorrespondence} and \autoref{Lbotprp} $T\mapsto T\cap T^*$ and $T\mapsto\sqrt{T}$ are order isomorphisms between $\mathscr{P}(S)^\mathrm{L}$ and $\mathscr{P}(S)^\bot$.  For any $T\in\mathscr{P}(S)^\mathrm{L}$, \autoref{Lbotprp} also yields \[T^\mathrm{L}=T^{\mathrm{L}\mathrm{L}\mathrm{L}}=T^{\mathrm{L}\bot\mathrm{L}}=(T^{\mathrm{L}\mathrm{L}}\cap T^{\mathrm{L}\mathrm{L}*})^\mathrm{L}=(T\cap T^*)^\mathrm{L}\] so $T^\bot=T^\mathrm{L}\cap T^{\mathrm{L}*}=(T\cap T^*)^\bot$, i.e. $T\mapsto T\cap T^*$ also preserves the orthocomplementation and $\mathscr{P}(S)^\mathrm{L}\cong\mathscr{P}(S)^\bot$ which, in particular, means $\mathscr{P}(S)^\bot$ is also a complete ortholattice in which infimums are intersections.  By the last part of \autoref{Lbotprp}, $T=T^{\triangledown\triangledown}=T^{\triangledown\bot}=T^{\triangledown\mathrm{L}}$, for any $T\in\mathscr{P}(S)^\triangledown$, i.e. $\mathscr{P}(S)^\triangledown\subseteq\mathscr{P}(S)^\mathrm{L}\cap\mathscr{P}(S)^\bot$.  While if $T\in\mathscr{P}(S)^\mathrm{L}$ and $T=T^*$ then $T$ is a right ideal and hence $T=T^{\mathrm{L}\mathrm{L}}=T^{\triangledown\mathrm{L}}=T^{\triangledown\triangledown}\in\mathscr{P}(S)^\triangledown$.  And if $T\in\mathscr{P}(S)^\bot$ and $ST\subseteq T$ then $T$ is again a right ideal so $T=T^{\bot\bot}=T^{\triangledown\bot}=T^{\triangledown\triangledown}\in\mathscr{P}(S)^\triangledown$.
\end{proof}

While the infimums above are nicely described by intersections, supremums are not in general described by unions or any other simple set theoretic or algebraic operations.  Indeed, to prove anything about expressions involving supremums often involves turning them into infimums, i.e. using the fact that $\bigvee\mathcal{T}=(\bigcap\mathcal{T}^\perp)^\perp$.  For example, to prove that supremums in $\mathscr{P}(S)^\triangledown$ agree with those in $\mathscr{P}(S)^\perp$ and $\mathscr{P}(S)^\mathrm{L}$, we use the fact that $T^\triangledown=T^\perp=T^\mathrm{L}$, for any $T\in\mathscr{P}(S)^\triangledown$, by \autoref{Lbotprp}, and so, for any $\mathcal{T}\subseteq\mathscr{P}(S)^\triangledown$, we have
\begin{equation}
\bigvee_{\mathscr{P}(S)^\triangledown}\mathcal{T}=(\bigcap\mathcal{T}^\triangledown)^\triangledown=(\bigcap\mathcal{T}^\perp)^\perp=\bigvee_{\mathscr{P}(S)^\perp}\mathcal{T}.
\end{equation}

If $S$ is (the multiplicative proper *-semigroup of) a Baer *-ring, then the correspondence between elements of $\mathscr{P}(S)^\bot$ and $\mathscr{P}(S)^\mathrm{L}$ with projections in $S$ does at least give a simple algebraic expression for supremums of \emph{orthogonal} $T$ and $U$ in either of these ortholattices, namely $T\vee U=T+U$.  But even this may fail for more general *-rings.  For example, if $S$ is a commutative C*-algebra then $S=C_0(X)$ for some topological space $X$, and elements of $\mathscr{P}(S)^\bot=\mathscr{P}(S)^\mathrm{L}$ correspond to regular open subsets of $X$.  For orthogonal $T$ and $U$, $T+U$ corresponds to the union of the corresponding regular open sets, but this may not be regular, and $T\vee U$ actually corresponds to the interior of the closure of this union (which is strictly larger than the union when it is not regular).

We could also have defined a relation $\Bot$ on any proper *-semigroup $S$ by \[s\Bot t\quad\Leftrightarrow\quad 0=st^*=st=s^*t=s^*t^*.\]  Then $S^\bot=S^{\sBot}$ whenever $S=S^*$, so in particular $^\bot$ and $^{\sBot}$ coincide on $\mathscr{P}(S)^\bot=\mathscr{P}(S)^{\sBot}$.  But $\Bot$ is symmetric (unlike $\bot$) so $\mathscr{P}(S)^{\sBot}$ is a complete ortholattice in which infimums are intersections, and thus the same applies to $\mathscr{P}(S)^\bot$, yielding another proof of this result.  However, we have chosen to focus on $\perp$ as this makes it slightly easier to define the equivalence relation in \S\ref{eq}.  Also, we could have defined a relation $\mathrm{R}$ on $S$ by \[s\mathrm{R}t\quad\Leftrightarrow\quad 0=s^*t\] and then everything we proved about $\mathrm{L}$ would hold for $\mathrm{R}$ with `right' replacing `left'.  And finally note that $\triangledown$ and $\mathrm{L}$ are really just two special cases of a more general relation $\mathrm{L}_T$ that one could define for any $T\subseteq S^1$ ($S^1$ denoting the unitization of $S$) by \[s\mathrm{L}_T t\quad\Leftrightarrow\quad sTt^*=\{0\},\] specifically $\triangledown=\mathrm{L}_S$ and $\mathrm{L}=\mathrm{L}_{\{1\}}$.  In general, $\mathrm{L}_T$ is symmetric if $T$ is self-adjoint in which case $\mathscr{P}(S)^{\mathrm{L}_T}$ will again be a complete ortholattice in which infimums are intersections.

\section{The *-Annihilator Ortholattice}\label{*AO}

Let us make the following standing assumption throughout the rest of this paper:
\begin{center}
\textbf{$S$ is a proper *-semigroup.}
\end{center}
Here we continue our investigation of the polarities we have defined in order to reveal structural properties of the ortholattice $\mathscr{P}(S)^\perp$ of \emph{*-annihilators}.  As far as purely ortholattice theoretic properties go, we could equally well deal with the ortholattice $\mathscr{P}(S)^\mathrm{L}$ of \emph{left annihilators}, by \autoref{Lperpequiv}, and if we were studying a category of semigroups rather than *-semigroups, this is precisely what we would do.  However, for us it is more natural to deal with $\mathscr{P}(S)^\perp$, because every *-annihilator is itself a proper *-semigroup, while left annihilators are generally not self-adjoint.

\begin{thm}\label{centhm}
If $A$ is a bi-hereditary *-subsemigroup of $S$ and $I$ is an ideal of $S$, \[A^\mathrm{L}=(A\cap I)^\mathrm{L}\cap(A\cap I^\mathrm{L})^\mathrm{L}\quad\textrm{and}\quad A^\perp=(A\cap I)^\perp\cap(A\cap I^\perp)^\perp.\]
\end{thm}

\begin{proof}
It suffices to show that \[S\setminus A^\mathrm{L}\subseteq S\setminus((A\cap I)^\mathrm{L}\cap(A\cap I^\mathrm{L})^\mathrm{L}).\]  Take $s\in S\setminus A^\mathrm{L}$, so $\exists a\in A_+$ such that $as^*\neq0$.  As $I^\mathrm{L}\cap I^{\mathrm{L}\mathrm{L}}=\{0\}$ we have $b\in I_+$ or $b\in I^\mathrm{L}_+$ with $abas^*\neq0$.  As $I$ and $I^\mathrm{L}=I^\triangledown$ are ideals, and $A$ is bi-hereditary, $aba\in A\cap I$ and hence $s\notin (A\cap I)^\mathrm{L}$, or $aba\in A\cap I^\mathrm{L}$ and hence $s\notin (A\cap I^\mathrm{L})^\mathrm{L}$.  In either case $s\notin(A\cap I)^\mathrm{L}\cap(A\cap I^\mathrm{L})^\mathrm{L}$.  This proves the first equation and the second now follows immediately from \autoref{Lbotprp}\eqref{Lbotprp3}.
\end{proof}

Recall that the \emph{centre} of an ortholattice $\mathbb{P}$ is the set of all $z\in\mathbb{P}$ such that $p=(p\wedge z)\vee(p\wedge z^\perp)$, for all $p\in\mathbb{P}$ (see \cite{MacLaren1964} \S3 or \cite{Kalmbach1983} \S3 Theorem 1).

\begin{cor}\label{annidealcentre}
The centre of $\mathscr{P}(S)^\perp$ contains $\mathscr{P}(S)^\triangledown$.
\end{cor}

\begin{proof}
For any $A\in\mathscr{P}(S)^\perp$ and $I\in\mathscr{P}(S)^\triangledown$, \autoref{centhm} yields
\[A=A^{\perp\perp}=((A\cap I)^\perp\cap(A\cap I^\perp)^\perp)^\perp=(A\wedge I)\vee(A\wedge I^\perp).\]
\end{proof}

We define a relation $\triangledown$ on $\mathscr{P}(S)$ from the function $^\triangledown$ in the standard way, namely \[T\triangledown U\quad\Leftrightarrow\quad T\leq U^\triangledown.\]  As $^\triangledown$ was itself defined from the relation $\triangledown$ on $S$, $T\triangledown U\Leftrightarrow\forall t\in T\forall u\in U(t\triangledown u)$.  And if $A,B,C\in\mathscr{P}(S)^\perp$ and $A\triangledown B$ then, as central elements are neutral (i.e. distributivity holds for expressions involving a central element \textendash\, see \cite{MaedaMaeda1970} Theorem (4.13)), $C\wedge B\leq(A\vee C)\wedge B\leq(A^{\triangledown\triangledown}\vee C)\wedge B=(A^{\triangledown\triangledown}\wedge B)\vee(C\wedge B)=C\wedge B$, i.e. \[A\triangledown B\quad\Rightarrow\quad\forall C\in\mathscr{P}(S)^\perp((A\vee C)\wedge B=C\wedge B).\]

This last statement expresses the standard `del' relation that can be defined in an arbitrary lattice $\mathbb{P}$ (see \cite{MaedaMaeda1970} Definition (4.3)), so our $\triangledown$ relation is at least as strong as the standard del relation.  By \cite{MaedaMaeda1970} Theorem (4.18), they coincide if $S$ is (the multiplicative *-semigroup of) a C*-algebra, as in this case $\mathscr{P}(S)^\perp$ is separative (and hence SSC and SSC* in the terminology of \cite{MaedaMaeda1970}), by \cite{Bice2014} Theorem 3.26.  In this case we also see that $\mathscr{P}(S)^\triangledown$ actually is the entire centre of $\mathscr{P}(S)^\perp$ (see \cite{Bice2014} Theorem 3.30) which, as every *-annihilator of a C*-algebra is a C*-algebra in its own right, means that $\mathscr{P}(S)^\perp$ has the relative centre property, by the following generalization of \cite{Chevalier1991} Proposition 14.

\begin{thm}\label{relcen}
If $A=A^*$ is a bi-ideal then $\mathscr{P}(A)^{\triangledown_A}\cong\mathscr{P}(A^{\triangledown\triangledown})\cap\mathscr{P}(S)^\triangledown$ via
\[J\mapsto J^{\triangledown\triangledown}\quad\textrm{and}\quad I\mapsto A\cap I.\]
\end{thm}

\begin{proof}
If $I\in\mathscr{P}(S)^\triangledown$, $a\in(A\cap I)^{\triangledown_A\triangledown_A}$, $b\in I^\triangledown$ and $s\in S$ then, as $I^\triangledown$ is an ideal and $A$ is bi-hereditary, $aa^*s^*b^*bsaa^*\in I^\triangledown\cap A\subseteq(A\cap I)^{\triangledown_A}$.  As $a\in(A\cap I)^{\triangledown_A\triangledown_A}$, $aa^*s^*b^*bsaa^*a=0$ and hence $bsa=0$.  As $b\in I^\triangledown$ and $s\in S$ were arbitrary, $a\in I^{\triangledown\triangledown}=I$.  As $a\in(A\cap I)^{\triangledown_A\triangledown_A}$ was arbitrary, $(A\cap I)^{\triangledown_A\triangledown_A}\subseteq A\cap I\subseteq (A\cap I)^{\triangledown_A\triangledown_A}$ so $A\cap I\in\mathscr{P}(A)^{\triangledown_A}$.

In particular, $A\cap I^\triangledown\in\mathscr{P}(A)^{\triangledown_A}$ and $A\cap I^\triangledown\subseteq(A\cap I)^{\triangledown_A}$.  By \autoref{centhm}
\[(A\cap I)^{\triangledown_A}\cap(A\cap I^\triangledown)^{\triangledown_A}\subseteq(A\cap I)^\perp\cap(A\cap I^\perp)^\perp\cap A=A^\perp\cap A=\{0\}.\]
By \autoref{annidealcentre} applied in $A$ we have $(A\cap I)^{\triangledown_A}=((A\cap I)^{\triangledown_A}\cap A\cap I^\triangledown)\vee\{0\}$ and hence $(A\cap I)^{\triangledown_A}\subseteq A\cap I^\triangledown$, i.e. the map $I\mapsto A\cap I$ preserves the orthocomplement.

Now let $B=I\cap(A\cap I)^\triangledown\in\mathscr{P}(S)^\triangledown$.  As $A\cap B=\{0\}$, we have $a^*ab^*ba^*a=0$, for all $a\in A$ and $b\in B$, and hence $ba^*=0$.  As $b\in B$ was arbitrary, we have $a\in B^\mathrm{L}=B^\triangledown$.  As $a\in A$ was arbitrary, $A\subseteq B^\triangledown$ so $B\subseteq A^\triangledown$.  Thus if $I\subseteq A^{\triangledown\triangledown}$ then, as $B\subseteq I$, $B=\{0\}$ and hence $I\subseteq(A\cap I)^{\triangledown\triangledown}(\subseteq I^{\triangledown\triangledown}=I)$, by \autoref{annidealcentre}, i.e. $J\mapsto J^{\triangledown\triangledown}$ is a left inverse for $I\mapsto A\cap I$.

Lastly, take $J\in\mathscr{P}(A)^{\triangledown_A}$.  We have already proved that $A\cap J^{\triangledown\triangledown}\in\mathscr{P}(A)^{\triangledown_A}$ and we immeidately see that $J\subseteq A\cap J^{\triangledown\triangledown}$.  By \autoref{triT}, $J^{\triangledown_A}\subseteq J^\triangledown$ so $A\cap J^{\triangledown\triangledown}\cap J^{\triangledown_A}\subseteq J^{\triangledown\triangledown}\cap J^\triangledown=\{0\}$ and hence $J=A\cap J^{\triangledown\triangledown}$, by \autoref{annidealcentre} applied in $A$, i.e. $J\mapsto J^{\triangledown\triangledown}$ is also a right inverse for $I\mapsto A\cap I$.
\end{proof}

Even if $A$ is just a bi-hereditary *-subsemigroup, rather than a bi-ideal, the above proof still shows that $I\mapsto A\cap I$ is an orthoisomorphism of $\mathscr{P}(A^{\triangledown\triangledown})\cap\mathscr{P}(S)^\triangledown$ onto a subset of $\mathscr{P}(A)^{\triangledown_A}$.  A similar argument yields the following.

\begin{thm}\label{comann}
If $A$ is a commutative bi-hereditary *-subsemigroup of $S$ then \[\mathscr{P}(A)^{\perp_A}=\{A\cap I:I\in\mathscr{P}(S)^\perp\}.\]
\end{thm}
\
\begin{proof}
If $I\in\mathscr{P}(S)^\perp$, $a\in(A\cap I)^{\perp_A\perp_A}$, $b\in I^\perp$ and $c\in A\cap I$ then we have $caa^*b^*baa^*=aa^*cb^*baa^*=0$.  As $c\in A\cap I$ was arbitrary, $aa^*b^*baa^*\in(A\cap I)^\perp$.  As $A$ is bi-hereditary, we also have $aa^*b^*baa^*\in A$ and hence $aa^*b^*baa^*\in(A\cap I)^{\perp_A}$.  As $a\in(A\cap I)^{\perp_A\perp_A}$, $aa^*b^*baa^*a=0$ and hence $ba=0$.  Likewise, $ba^*=0$ which, as $b\in I^\perp$ was arbitrary, means that $a\in I^{\perp\perp}=I$.  As $a\in(A\cap I)^{\perp_A\perp_A}$ was arbitrary, $A\cap I=(A\cap I)^{\perp_A\perp_A}\in\mathscr{P}(A)^{\perp_A}$.

On the other hand, if $J\in\mathscr{P}(A)^{\perp_A}=\mathscr{P}(A)^{\triangledown_A}$ then, by the previous paragraph, $A\cap J^{\perp\perp}\in\mathscr{P}(A)^{\perp_A}$ and $A\cap J^{\perp\perp}\cap J^{\perp_A}\subseteq J^{\perp\perp}\cap J^\perp=\{0\}$.  Applying \autoref{annidealcentre} in $A$, we have $J=A\cap J^{\perp\perp}\in\{A\cap I:I\in\mathscr{P}(S)^\perp\}$.
\end{proof}

Note that when the $A$ above is an *-annihilator, the result is saying that \[\mathscr{P}(A)^{\perp_A}=[\{0\},A],\] where the interval $[\{0\},A]$ is taken in $\mathscr{P}(S)^\perp$.  There are examples of proper *-semigroups $S$ (even C*-algebras - see \cite{Bice2014} Example 3.87) such that $\mathscr{P}(S)^\perp$ is not orthomodular, in which case there are intervals $[\{0\},A]$ that are not ortholattices (w.r.t. $^{\perp_A}$).  So the commutativity assumption is crucial in \autoref{comann}.

\begin{dfn}
We call $A\in\mathscr{P}(S)^\perp$ \emph{$\triangledown$-finite} if, for all $B\in\mathscr{P}(S)^\perp$ with $B\subseteq A$,
\[B^\triangledown=A^\triangledown\quad\Rightarrow\quad B=A.\]
\end{dfn}

As $B^\triangledown=A^\triangledown\Leftrightarrow B^{\triangledown\triangledown}=A^{\triangledown\triangledown}$ and $A^{\triangledown\triangledown}=\bigcap_{A\subseteq I\in\mathscr{P}(S)^\triangledown}I$, for all $A,B\subseteq S$, we see that the $\triangledown$-finite *-annihilators are precisely the $\sim_{\mathscr{P}(S)^\triangledown}$-finite elements of the ortholattice $\mathscr{P}(S)^\perp$, according to \cite{Bice2014b} Definition 4.3 and 4.5.

\begin{cor}
If $A\in\mathscr{P}(S)^\perp$ is commutative then $A$ is $\triangledown$-finite.
\end{cor}

\begin{proof} By \autoref{relcen} and \autoref{comann}, we have
\[[\{0\},A]=\mathscr{P}(A)^{\perp_A}=\mathscr{P}(A)^{\triangledown_A}=\{A\cap I:I\in\mathscr{P}(S)^\triangledown\}.\]
So if $B\in[\{0\},A]$ and $B^\triangledown=A^\triangledown$ then $B=A\cap B^{\triangledown\triangledown}=A\cap A^{\triangledown\triangledown}=A$.
\end{proof}

The converse also holds when $S$ is (the multiplicative *-semigroup of) a C*-algebra, i.e. in this case commutative and $\triangledown$-finite *-annihilators coincide (see \cite{Bice2014} Theorem 3.49).  They are also given various other names in slightly different contexts, like monad/lowest/simple/Boolean/D-element (see the comments after \cite{Bice2014b} Theorem 4.8).

The important thing about $\triangledown$-finite elements is that they immediately yield type decompositions.  For example, from \cite{Bice2014b} Theorem 2.6 (together with \cite{Bice2014b} Proposition 4.4 and Theorem 4.6), we get the following result.

\begin{thm}\label{typeI}
There is a unique $A\in\mathscr{P}(S)^\triangledown$ such that $A=B^{\triangledown\triangledown}$, for some $\triangledown$-finite $B\in\mathscr{P}(S)^\perp$, while $A^\perp$ contains no $\triangledown$-finite *-annihilators whatsoever.
\end{thm}

When $S$ is a von Neumann algebra the $A$ above is the type $\mathrm{I}$ part of $S$, while $A^\perp$ consists of the type $\mathrm{II}$ and $\mathrm{III}$ parts of $S$.  Likewise, using \cite{Bice2014b} Theorem 2.4, we get the following decomposition which, when $S$ is a von Neumann algebra, is just the usual decomposition into abelian (type $\mathrm{I}_1$) and properly non-abelian parts.

\begin{thm}\label{typeI1}
There is a unique $\triangledown$-finite $A\in\mathscr{P}(S)^\triangledown$ such that $A^\perp$ contains no $\triangledown$-finite elements of $\mathscr{P}(S)^\triangledown$.
\end{thm}

While on the topic of type decomposition, let us point out that such decompositions have traditionally been obtained with respect to the entire centre, rather than a subset like $\mathscr{P}(S)^\triangledown$ here.  The problem with using the centre of $\mathscr{P}(S)^\perp$ here is that it may not be complete.  Indeed, previous type decomposition results have used an assumption like orthomodularity or separativity to show that the centre of a complete lattice is complete, assumptions which may not hold for $\mathscr{P}(S)^\perp$.  However $\mathscr{P}(S)^\triangledown$, being obtained from polarity, is automatically a complete lattice, which allows us to always obtain type decompositions with respect to $\mathscr{P}(S)^\triangledown$.

We finish this section with some results showing that essential ideals are essentially the same as the entire *-semigroup, at least when it comes to *-annihilators.

\begin{dfn}
We call $T\subseteq S$ \emph{essential} if $T^\perp=\{0\}$.
\end{dfn}

\begin{prp}\label{ILTL}
If $I$ is an essential right ideal of a bi-hereditary *-subsemigroup $A\subseteq S$ then $I^\mathrm{L}=A^\mathrm{L}$, $I^\perp=A^\perp$ and $I^\triangledown=A^\triangledown$.
\end{prp}

\begin{proof}
If $s\in S$, $a\in A$, $u\in I$ and $v\in I^\triangledown$ then $ua^*a\in I$ and hence $ua^*asv=0=ua^*asvv^*s^*a^*a$.  As $u\in I$ was arbitrary, $a^*asvv^*s^*a^*a\in A\cap I^\perp=I^{\perp_A}=\{0\}$ so $asv=0$.  As $a\in A$ and $s\in S$ were arbitrary, $v\in A^\triangledown$.  As $v\in I^\triangledown$ was arbitrary, $I^\triangledown\subseteq A^\triangledown$ while $A^\triangledown\subseteq I^\triangledown$ is immediate from $I\subseteq A$.  This proves the last equality, and the others follow by essentially the same argument, just without the $s\in S$.
\end{proof}

\begin{prp}\label{IcapT}
If $I$ is an essential ideal of $S$ and $A$ is a bi-hereditary *-subsemigroup of $S$ then $A\cap I$ is essential in $A$.
\end{prp}

\begin{proof}
For any $a\in A\setminus\{0\}$ we have $s\in I$ with $sa\neq0$ so $aa^*s^*saa^*a\neq0$ and $aa^*s^*saa^*\in A\cap I$ so $a\notin(A\cap I)^{\perp_A}$, i.e. $(A\cap I)^{\perp_A}=\{0\}$.
\end{proof}

\begin{thm}\label{ess}
If $I=I^*$ is an essential ideal of $S$ then $B\mapsto B^{\perp\perp}$ and $A\mapsto A\cap I$ are mutually inverse orthoisomorphisms witnessing \[\mathscr{P}(I)^{\perp_I}\cong\mathscr{P}(S)^\perp\quad\textrm{and}\quad\mathscr{P}(I)^{\triangledown_I}\cong\mathscr{P}(S)^\triangledown.\]
\end{thm}

\begin{proof}
If $B\in\mathscr{P}(I)^{\perp_I}$ then $B^{\perp_I\perp}=(B^\perp\cap I)^\perp=B^{\perp\perp}$, by \autoref{ILTL} and \autoref{IcapT}.  Thus $B=B^{\perp_I\perp_I}=B^{\perp_I\perp}\cap I=B^{\perp\perp}\cap I$ and $(B^{\perp_I})^{\perp\perp}=(B^{\perp\perp})^\perp$.  On the other hand, if $A\in\mathscr{P}(S)^\perp$ then $(A\cap I)^\perp=A^\perp$, again by \autoref{ILTL} and \autoref{IcapT}.  Thus $A=A^{\perp\perp}=(A\cap I)^{\perp\perp}$ and $A^\perp\cap I=(A\cap I)^\perp\cap I=(A\cap I)^{\perp_I}$.  So the given maps do indeed witness $\mathscr{P}(I)^{\perp_I}\cong\mathscr{P}(S)^\perp$ and essentially the same argument shows that they also witness $\mathscr{P}(I)^{\triangledown_I}\cong\mathscr{P}(S)^\triangledown$.
\end{proof}

\section{*-Equivalence}\label{eq}

The present section is devoted to the binary relation $\sim$ on $\mathscr{P}(S)^{\perp}$ defined by
\[A\sim B\quad\Leftrightarrow\quad\exists s\in S(\{s\}^{\perp\perp}=A\textrm{ and }\{s^*\}^{\perp\perp}=B).\]
The importance of $\sim$ can be seen from the fact it extends Murray-von Neumann equivalence for projections in a von Neumann algebra $S$, i.e. the binary relation $\sim_\mathrm{MvN}$ on $S_\mathrm{proj}$ defined by
\[p\sim_\mathrm{MvN}q\quad\Leftrightarrow\quad\exists s\in S(s^*s=p\textrm{ and }ss^*=q).\]
In fact, we show in \autoref{MvN} that $\sim$ generalizes $\sim_\mathrm{MvN}$ for certain proper *-semigroups $S$ satisfying \emph{polar decomposition}, meaning that, for all $a\in S$, we have $b\in S_\mathrm{sa}$ with $a^*a=b^2$ and $a\in Sb$ (this is a weakening of \cite{Berberian1972} \S21 Definition 1).  As is well-known, von Neumann algebras have polar decomposition, but so too do AW*-algebras, Rickart C*-algebras (see \cite{AraPere1993}) and all quotients of these (e.g. the Calkin algebra).  These C*-algebras also satisfy the following (see \autoref{add}).

\begin{dfn}
A *-semigroup $S$ is \emph{$\perp$-cancellative} if, for all $a,b,s,t\in S$,
\[\{a\}^\perp=\{b\}^\perp\textrm{ and }as=at\quad\Rightarrow\quad bs=bt.\]
\end{dfn}

\begin{lem}\label{appb}\label{simlem}
For $a,b\in S$, we have $\{ab\}^{\perp\perp}=(\{a\}^{\perp\perp}b)^{\perp\perp}$ and if $\{a\}^{\perp\perp}\subseteq\{b^*\}^{\perp\perp}$ then $\{b^*a^*\}^{\perp\perp}=\{a^*\}^{\perp\perp}$, i.e. $ab$ witnesses $\{a^*\}^{\perp\perp}\sim(\{a\}^{\perp\perp}b)^{\perp\perp}$.
\end{lem}

\begin{proof}
If $s\in\{ab\}^\perp_+$ then $absb^*=0$ so we have $bsb^*\in\{a\}^\perp=(\{a\}^{\perp\perp})^\perp$ and hence $s\in(\{a\}^{\perp\perp}b)^\perp$.  As $s\in\{ab\}^\perp_+$ was aribitrary, $\{ab\}^\perp\subseteq(\{a\}^{\perp\perp}b)^\perp$.  As $a^*a\in\{a\}^{\perp\perp}$, \autoref{s*st} yields $(\{a\}^{\perp\perp}b)^\perp\subseteq\{a^*ab\}^\perp=\{ab\}^\perp$ so we have $(\{a\}^{\perp\perp}b)^{\perp\perp}=\{ab\}^{\perp\perp}$.

Now if $\{a\}^{\perp\perp}\subseteq\{b^*\}^{\perp\perp}$ and $s\in\{b^*a^*\}^{\perp}_+$ then $b^*a^*sa=0$ so \[a^*sa\in\{b^*\}^\perp\cap\{a\}^{\perp\perp}\subseteq\{b^*\}^\perp\cap\{b^*\}^{\perp\perp}=\{0\}.\]  Thus $a^*s=0$ and $s\in\{a^*\}^\perp$.  As $s\in\{b^*a^*\}^{\perp}_+$ was arbitrary, $\{b^*a^*\}^\perp\subseteq\{a^*\}^\perp$, while $\{a^*\}^\perp\subseteq\{b^*a^*\}^\perp$ is immediate, so $\{b^*a^*\}^{\perp\perp}=\{a^*\}^{\perp\perp}$.
\end{proof}

We can now give the following generalization of \cite{Berberian1972} \S21 Proposition 3.

\begin{prp}\label{MvN}
If $S$ is a proper $\perp$-cancellative *-semigroup with polar decomposition and $r\mapsto\{r\}^{\perp\perp}$ is injective on $S_\mathrm{proj}$ then, for all $p,q\in S_\mathrm{proj}$, \[p\sim_\mathrm{MvN}q\quad\Leftrightarrow\quad\{p\}^{\perp\perp}\sim\{q\}^{\perp\perp}.\]
\end{prp}

\begin{proof}
If $s^*s=p$ and $ss^*=q$ then $\{s\}^{\perp\perp}=\{s^*s\}^{\perp\perp}=\{p\}^{\perp\perp}$ and $\{s^*\}^{\perp\perp}=\{ss^*\}^{\perp\perp}=\{q\}^{\perp\perp}$, i.e. if $s$ witnesses $p\sim_\mathrm{MvN}q$, it also witnesses $\{p\}^{\perp\perp}\sim\{q\}^{\perp\perp}$.

Now say $\{s\}^{\perp\perp}=\{p\}^{\perp\perp}$ and $\{s^*\}^{\perp\perp}=\{q\}^{\perp\perp}$.  By polar decomposition, we have $u\in S$ and $t\in S_\mathrm{sa}$ with $s^*s=t^2$ and $s=ut$.  As $tu^*ut=s^*s=t^2$ and $\{t\}^\perp=\{t^2\}^\perp=\{s^*s\}^\perp=\{s\}^\perp=\{p\}^\perp$, $\perp$-cancellativity yields $p=pp=pu^*up=w^*w$, where $w=up$.  Thus $w^*ww^*=w^*$ and hence $ww^*ww^*=ww^*$ so $ww^*\in S_\mathrm{proj}$.  By \autoref{appb}, $\{q\}^{\perp\perp}=\{s^*\}^{\perp\perp}=\{tu^*\}^{\perp\perp}=\{\{t\}^{\perp\perp}u^*\}^{\perp\perp}=\{\{p\}^{\perp\perp}u^*\}^{\perp\perp}=\{pu^*\}^{\perp\perp}=\{w^*\}^{\perp\perp}=\{ww^*\}^{\perp\perp}$ so the injectivity of $r\mapsto\{r\}^{\perp\perp}$ on $S_\mathrm{proj}$ gives $q=ww^*$, i.e. $w$ witnesses $p\sim_\mathrm{MvN}q$.
\end{proof}

Next we show that $\sim$ shares some well-known properties of $\sim_\mathrm{MvN}$ even in much more general *-semigroups.  The first fundamental such property is transitivity.

Note we write $A\precsim B$ if $A\sim C\subseteq B$, for some $C\in\mathscr{P}(S)^{\perp}$.

\begin{thm}\label{simtran}
The relations $\sim$ and $\precsim$ are transitive.  They are reflexive iff
\begin{equation}\label{simtraneq}
\mathscr{P}(S)^\perp=\{\{s\}^\perp:s\in S\}.
\end{equation}
\end{thm}

\begin{proof}
Say $a$ witnesses $A\precsim B$, so $\{a\}^{\perp\perp}=A$ and $\{a^*\}^{\perp\perp}\subseteq B$, and $b$ witnesses $B\precsim C$, so $\{b\}^{\perp\perp}=B$ and $\{b^*\}^{\perp\perp}\subseteq C$.  Then $\{a^*\}^{\perp\perp}\subseteq\{b\}^{\perp\perp}$ and hence $\{ba\}^{\perp\perp}=\{a\}^{\perp\perp}=A$, by \autoref{simlem}, while $\{a^*b^*\}^{\perp\perp}\subseteq\{b^*\}^{\perp\perp}=C$, i.e. $A\precsim C$.  If $a$ and $b$ actually witnessed $A\sim B\sim C$ then again by \autoref{simlem} we would have $\{a^*b^*\}^{\perp\perp}=\{b^*\}^{\perp\perp}=C$, i.e. $A\sim C$.

If $\precsim$ is reflexive then, for every $A\in\mathscr{P}(S)^\perp$, we have $A^\perp\precsim A^\perp$ so $A^\perp=\{s\}^{\perp\perp}$, for some $s\in S$, and hence $A=A^{\perp\perp}=\{s\}^\perp$.  On the other hand, if $A^\perp=\{s\}^\perp$ then $A=A^{\perp\perp}=\{s^*s\}^{\perp\perp}=\{(s^*s)^*\}^{\perp\perp}$, i.e. $A\sim A$, so \eqref{simtraneq} implies that $\sim$, and hence $\precsim$, is reflexive.
\end{proof}

Just as any partial isometry yields a map between the projection lattices below its support projections, any $s\in S$ yields a map on the *-annihilator ortholattice.

\begin{lem}\label{nonorthodiv}
For $s\in S$, a supremum preserving map on $\mathscr{P}(S)^\perp$ is given by \[A\mapsto(As)^{\perp\perp}.\]
\end{lem}

\begin{proof}
Given $\mathcal{A}\subseteq\mathscr{P}(S)^\perp$ we have \[\bigvee_{A\in\mathcal{A}}(As)^{\perp\perp}=(\bigcap_{A\in\mathcal{A}}(As)^\perp)^\perp=((\bigcup\mathcal{A})s)^{\perp\perp}\subseteq((\bigvee\mathcal{A})s)^{\perp\perp}.\] To see that this last inclusion can be reversed, take $a\in((\bigcup\mathcal{A})s)^\perp_+$.  This means $sas^*\in(\bigcup\mathcal{A})^\perp=(\bigvee\mathcal{A})^\perp$ and hence $a\in((\bigvee\mathcal{A})s)^\perp$, i.e. $((\bigcup\mathcal{A})s)^\perp\subseteq((\bigvee\mathcal{A})s)^\perp$ and hence $((\bigvee\mathcal{A})s)^{\perp\perp}\subseteq((\bigcup\mathcal{A})s)^{\perp\perp}$.
\end{proof}

However, $A\mapsto(As)^{\perp\perp}$ may not preserve infimums, as the following example shows.  Specifically, consider $S=\mathcal{B}(H)$, where $H$ is a separable infinite dimensional Hilbert space with basis $(e_n)$.  Define $s\in\mathcal{B}(H)_+$ by $s(e_n)=\frac{1}{n^2}e_n$, so $\overline{\mathcal{R}(s)}=H$ and hence $\{s\}^{\perp\perp}=S$, and let $p$ be the projection onto $\mathbb{C}v$, where $v=\sum\frac{1}{n}e_n$.  Then $q=1-p\perp p$, and hence $qSq=\{q\}^{\perp\perp}\perp\{p\}^{\perp\perp}=pSp$ even though we still have $\overline{\mathcal{R}(sq)}=H$ and hence $(qSqs)^{\perp\perp}=\{qs\}^{\perp\perp}=\{s\}^{\perp\perp}=S$, while $(pSps)^{\perp\perp}\neq\{0\}$.

This example also shows that $A\mapsto(As)^{\perp\perp}$ may not preserve orthogonality, so we can not use it to show that $\sim$ is orthogonally divisible (i.e. divisible in the usual dimension relation sense \textendash\, \cite{Kalmbach1983} \S11 Definition 1).  However, it does yield a weaker (possibly non-orthogonal) version of complete divisibility.

\begin{thm}\label{div}
If $\sim$ is reflexive, $\mathcal{A}\subseteq\mathscr{P}(S)^\perp$ and $\bigvee\mathcal{A}\sim B\in\mathscr{P}(S)^\perp$ then we have $f:\mathcal{A}\rightarrow\mathscr{P}(S)^\perp$ with $B=\bigvee_{A\in A}f(A)$ and $A\sim f(A)$, for all $A\in\mathcal{A}$.
\end{thm}

\begin{proof}
Take $s\in S$ with $\{s^*\}^{\perp\perp}=\bigvee\mathcal{A}$ and $\{s\}^{\perp\perp}=B$.  For each $A\in\mathcal{A}$ let $f(A)=(As)^{\perp\perp}$.  By \autoref{appb} and \autoref{nonorthodiv},
\[\bigvee_{A\in A}f(A)=((\bigvee\mathcal{A})s)^{\perp\perp}=\{s^*s\}^{\perp\perp}=B.\]
As $\sim$ is reflexive, for each $A\in\mathcal{A}$ we have $a\in A_+$ with $\{a\}^{\perp\perp}=A$ so $\{s^*a\}^{\perp\perp}=A$ and $\{as\}^{\perp\perp}=(As)^{\perp\perp}=f(A)$, by \autoref{appb}, i.e. $as$ witnesses $A\sim f(A)$.
\end{proof}

\begin{prp}\label{simcen}
If $A\precsim B$ then $B^\triangledown\subseteq A^\triangledown$ and if $A\sim B$ then $A^\triangledown=B^\triangledown$.
\end{prp}

\begin{proof}
If $A\precsim B$ then we have $s\in S$ with $\{s\}^{\perp\perp}=A$ and $\{s^*\}^{\perp\perp}\subseteq B$.  If $b\in B^\triangledown$ then, for any $a\in A$ and $t\in S$ we have $ss^*\in B$ so $ss^*satb=0=satb$ and hence $atbb^*t^*a^*\in\{s\}^\perp\cap A=A^\perp\cap A=\{0\}$, i.e. $atb=0$.  As $a\in A$ and $t\in S$ were arbitrary, $b\in A^\triangledown$.  As $b\in B^\triangledown$ was arbitrary, $B^\triangledown\subseteq A^\triangledown$.  Now if $A\sim B$ then certainly $A\precsim B\precsim A$ so $A^\triangledown=B^\triangledown$.
\end{proof}

\begin{cor}\label{cendiv}
If $\sim$ is reflexive then, for all $A,B\in\mathscr{P}(S)^\perp$ and $I\in\mathscr{P}(S)^\triangledown$,
\[A\sim B\ \Rightarrow\ A\cap I\sim B\cap I\qquad\textrm{and}\qquad A\precsim B\ \Rightarrow\ A\cap I\precsim B\cap I.\]
\end{cor}

\begin{proof}
By \autoref{annidealcentre}, $A=(A\cap I)\vee(A\cap I^\perp)$.  If $A\sim B$ then, by \autoref{div}, we have $C,D\in\mathscr{P}(S)^\perp$ with $A\cap I\sim C$, $A\cap I^\perp\sim D$ and $B=C\vee D$.  By \autoref{simcen}, $C\subseteq C^{\triangledown\triangledown}=(A\cap I)^{\triangledown\triangledown}\subseteq I$ and likewise $D\subseteq I^\perp$.  As central elements are neutral, $B\cap I=(C\cap I)\vee(D\cap I)=C$ and hence $A\cap I\sim B\cap I$.

Similarly, if $A\precsim B$ then, by \autoref{div}, we have $C\in\mathscr{P}(S)^\perp$ with $A\cap I\sim C\subseteq B$.  Again $C\subseteq I$ so $A\cap I\sim C=C\cap I\subseteq B\cap I$, i.e. $A\cap I\precsim B\cap I$.
\end{proof}

This corollary, together with a suitable extra assumption, allows us to get type decompositions from $\sim$, just as before we got type decompositions from $\triangledown$.

\begin{dfn}
Given binary relations $R$ and $\approx$ on a poset $\mathbb{P}$, we call $\mathcal{A}\subseteq\mathbb{P}$ an \emph{$R$-subset} if $\mathcal{A}\times\mathcal{A}\subseteq R\ \cup=$, i.e. $pRq$, for all distinct $p,q\in\mathcal{A}$, and we call $\approx$
\begin{itemize}
\item \emph{$R$-additive} if $p\vee q\approx p'\vee q'$ whenever $pRq$, $p'Rq'$, $p\approx p'$ and $q\approx q'$.
\item \emph{$R$-complete} if $\bigvee\mathcal{A}\approx\bigvee\mathcal{B}$ whenever $\mathcal{A}$ and $\mathcal{B}$ are $R$-subsets and there is a bijection $f:\mathcal{A}\rightarrow\mathcal{B}$ with $p\approx f(p)$, for all $p\in\mathcal{A}$.
\end{itemize}
\end{dfn}

We also call $A\in\mathscr{P}(S)^\perp$, \emph{$\sim$-finite} if $A\sim B\subseteq A\Rightarrow A=B$, for all $B\in\mathscr{P}(S)^\perp$.

\begin{thm}\label{typeIII}
If $\sim$ is reflexive and $\triangledown$-complete, there is a unique $A\in\mathscr{P}(S)^\triangledown$ with $A=B^{\triangledown\triangledown}$, for a $\sim$-finite $B\in\mathscr{P}(S)^\perp$, while $A^\perp$ contains no $\sim$-finite *-annihilators.
\end{thm}

\begin{proof}
Follows from \autoref{cendiv} and \cite{Bice2014b} Theorem 2.6 and Proposition 4.4.
\end{proof}

When $S$ is a von Neumann algebra the $A$ above consists of the type $\mathrm{I}$ and type $\mathrm{II}$ parts of $S$, while $A^\perp$ is the type $\mathrm{III}$ part of $S$.  Likewise, using \cite{Bice2014b} Theorem 2.4, we get the following decomposition which, when $S$ is a von Neumann algebra, is just the usual decomposition into finite (the type $\mathrm{I}_n$ parts, for $n<\infty$, and type $\mathrm{II}_1$ part) and properly infinte parts (the type $\mathrm{II}_\infty$ and type $\mathrm{III}$ parts).

\begin{thm}\label{typefin}
If $\sim$ is reflexive and $\triangledown$-complete then there is a unique $\sim$-finite $A\in\mathscr{P}(S)^\triangledown$ such that $A^\perp$ contains no $\sim$-finite elements of $\mathscr{P}(S)^\triangledown$.
\end{thm}

Next we want to show that $\sim$ satisfies the parallelogram law (see \cite{Berberian1972} \S13) or at least something equivalent in the orthomodular case, namely that $\sim$ is weaker than perspectivity (recall that elements $p$ and $q$ in a bounded poset $\mathbb{P}$ are \emph{perspective} if they have a common complement $r\in\mathbb{P}$, i.e. $p\wedge r=0=q\wedge r$ and $p\vee r=1=q\vee r$).

\begin{lem}\label{posp'}
If $a,b\in S$ and $\{a\}^\perp\cap\{b\}^{\perp\perp}=\{0\}=\{a\}^{\perp\perp}\cap\{b\}^\perp$ then \[\{a\}^{\perp\perp}\sim\{b\}^{\perp\perp}.\]
\end{lem}

\begin{proof}
By replacing $a$ and $b$ with $a^*a$ and $b^*b$ if necessary, we may assume that $a,b\in S_+$.  Given $s\in S_+$ with $abs=0$ we have $bsb\in \{b\}^{\perp\perp}\cap\{a\}^\perp=\{0\}$ and hence $bs=0$, i.e. $\{ab\}^{\perp}=\{b\}^\perp$ and hence $\{ab\}^{\perp\perp}=\{b\}^{\perp\perp}$.  A symmetric argument gives $\{ba\}^{\perp\perp}=\{a\}^{\perp\perp}$ and hence $ab$ witnesses $\{a\}^{\perp\perp}\sim\{b\}^{\perp\perp}$.
\end{proof}

\begin{thm}\label{simper}
If $\sim$ is reflexive then $\sim$ is weaker than perspectivity.
\end{thm}

\begin{proof}
If $A,B\in\mathscr{P}(S)^\perp$ have a common complement $C\in\mathscr{P}(S)^\perp$ then, by \autoref{posp'}, $A\sim C^\perp\sim B$ and hence, by \autoref{simtran}, $A\sim B$.
\end{proof}

\begin{thm}
If $\sim$ is reflexive, every $A\in\mathscr{P}(S)^\perp$ is $\sim$-finite and $\mathscr{P}(S)^\triangledown$ is the entire centre of $\mathscr{P}(S)^\perp$ then $\mathscr{P}(S)^\perp$ is modular and $\sim$ coincides with perspectivity.
\end{thm}

\begin{proof}
As $\sim$ is finite on all of $\mathbb{P}=\mathscr{P}(S)^\perp$, so is perspectivity, and any ortholattice $\mathbb{P}$ in which perspectivity is finite must be modular, as we now show.  Firstly, $\mathbb{P}$ must be orthomodular because it can not contain a copy of the benzene ring (see \cite{Beran1985} Fig. 7a).  By \cite{Jacobson1985} Theorem 8.4, modularity is equivalent to
\begin{equation*}\label{persp}
p\vee r=q\vee r,\ p\wedge r=q\wedge r\textrm{ and }p\leq q\quad\Rightarrow\quad p=q,
\end{equation*}
for all $p,q,r\in\mathbb{P}$.  Given $p,q,r\in\mathbb{P}$ with $p\vee r=q\vee r$ and $p\wedge r=q\wedge r$ we can define $p'=p\wedge(p\wedge r)^\perp$ and $p''=p'\vee(p'\vee r)^\perp$ and likewise for $q'$ and $q''$.  Then $p''$ and $q''$ have $r$ as a complement, and if $p\leq q$ then $p''\leq q''$ so $p''=q''$, as perspectivity is finite.  Orthomodularity now implies $p'=q'$ and $p=q$, and hence $\mathbb{P}$ is modular.

As $\mathbb{P}$ is complete, it is a continuous geometry, by \cite{Kaplansky1955}.  Now we follow the proof of \cite{Kaplansky1951} Theorem 6.6(c).  Specifically, by \cite{vonNeumann1960} Part III Theorem 2.7, for any $q,r\in\mathbb{P}$, we have central $p\in\mathbb{P}$ with $p\wedge q$ perspective to some $s\leq p\wedge r$ and $p^\perp\wedge r$ perspective to some $t\leq p^\perp\wedge q$.  So if $q\sim r$ then $s\sim p\wedge q\sim p\wedge r$, by \autoref{cendiv} and the fact $\mathscr{P}(S)^\triangledown$ is the centre of $\mathbb{P}$.  As $\sim$ is transitive and finite, $s=p\wedge r$ and, likewise, $t=p^\perp\wedge q$.  So $p\wedge q$ and $p\wedge r$ are perspective, as are $p^\perp\wedge q$ and $p^\perp\wedge r$ which, as $p$ is central, means $q$ and $r$ are also perspective.
\end{proof}

We also have an analog of the Cantor-Schroeder-Bernstein theorem.

\begin{thm}\label{CSB}
If $\sim$ is reflexive and $\perp$-additive then, for all $A,B\in\mathscr{P}(S)^\perp$, \[A\precsim B\precsim A\quad\Rightarrow\quad A\sim B.\]
\end{thm}

\begin{proof}
We adapt the argument given in \cite{Berberian1972} \S1 Theorem 1.  Specifically, take $a,b\in S$ with $\{a\}^{\perp\perp}=A$, $\{a^*\}^{\perp\perp}\subseteq B$, $\{b\}^{\perp\perp}=B$ and $\{b^*\}^{\perp\perp}\subseteq A$.  By \autoref{nonorthodiv}, we have an order preserving map on $\mathscr{P}(A)^{\perp_A}$ defined by \[C\mapsto((Ca^*)^{\perp_B}b^*)^{\perp_A}.\] As $\mathscr{P}(A)^{\perp_A}$ is a complete lattice, it has a fixed point $F$, by Tarski's fixed point theorem.  By \autoref{appb},
\begin{eqnarray*}
F^{\perp_A}=((Fa^*)^{\perp_B}b^*)^{\perp_A\perp_A} &\sim& (Fa^*)^{\perp_B}\qquad\textrm{and}\\
F &\sim& (Fa^*)^{\perp_B\perp_B}\qquad\textrm{so, by $\perp$-additivity,}\\
F\vee F^{\perp_A} &\sim& (Fa^*)^{\perp_B\perp_B}\vee(Fa^*)^{\perp_B}.
\end{eqnarray*}
While $F\vee_AF^{\perp_A}=A$, we may have $F\vee F^{\perp_A}<A$, but in any case $A^\perp$ is a common complement of $F\vee F^{\perp_A}$ and $A$ in $\mathscr{P}(S)^\perp$ so, by \autoref{simper}, $A\sim F\vee F^{\perp_A}$.  Likewise, $B\sim(Fa^*)^{\perp_B\perp_B}\vee(Fa^*)^{\perp_B}$ and hence $A\sim B$, by \autoref{simtran}.
\end{proof}

We now show $\precsim$ satisfies a version of generalized comparability (see \cite{Berberian1972} \S14).

\begin{thm}\label{gencom}
If $\sim$ is reflexive and $\perp$-complete then, for all $A,B\in\mathscr{P}(S)$, we have $I\in\mathscr{P}(S)^\triangledown$ such that $A\cap I\precsim B\cap I$ and $B\cap I^\perp\precsim A\cap I^\perp$.
\end{thm}

\begin{proof}
Let $(A_\lambda),(B_\lambda)\subseteq\mathscr{P}(S)^\perp$ be a maximal $\perp$-subsets of $A$ and $B$ respectively such that $A_\lambda\sim B_\lambda$, for all $\lambda\in\Lambda$.  As $\sim$ is $\perp$-complete, $\bigvee A_\lambda\sim\bigvee B_\lambda$.  As $A^\perp\vee(\bigvee A_\lambda)^{\perp_A}$ is a common complement of $\bigvee A_\lambda$ and $C=(\bigvee A_\lambda)^{\perp_A\perp_A}$, we have $C\sim\bigvee A_\lambda$, by \autoref{simper}.  Likewise, $D=(\bigvee B_\lambda)^{\perp_B\perp_B}\sim\bigvee B_\lambda$ and hence $C\sim D$, by \autoref{simtran}.  By maximality, we must have $C^{\perp_A}SD^{\perp_B}=\{0\}$ and thus taking $I=C^{\perp_A\triangledown}$ or $I=D^{\perp_B\triangledown\triangledown}$ we have $C^{\perp_A}\subseteq I^\perp$ and $D^{\perp_B}\subseteq I$.  Therefore, $A\cap I\subseteq C^{\perp_A\perp_A}=C$ so, by \autoref{cendiv}, $A\cap I=C\cap I\sim D\cap I\subseteq B\cap I$, i.e. $A\cap I\precsim B\cap I$.  A symmetric argument also yields $B\cap I^\perp\sim A\cap I^\perp$.
\end{proof}

If $I$ is a *-subsemigroup of $S$, we write $\sim_I$ for the $\sim$ relation defined within $I$.

\begin{cor}\label{simBsim}
If $\sim$ is reflexive and $A\in\mathscr{P}(S)^\perp$, $\sim_A$ is $\sim$ restricted to $\mathscr{P}(A)^{\perp_A}$.
\end{cor}

\begin{proof}
If $a\in S$ and $\{a\}^{\perp\perp}\in\mathscr{P}(A)^{\perp_A}$ then $\{a\}^{\perp_A\perp_A}=\{a\}^{\perp\perp\perp_A\perp_A}=\{a\}^{\perp\perp}$.  So if $a$ witnesses $B\sim C$, for $B,C\in\mathscr{P}(A)^{\perp_A}$, then it also witnesses $B\sim_AC$ as long as $a\in A$, which holds because $a^*a,aa^*\in A$ and *-annihilators are quasi-rooted.

On the other hand, for any $a\in A$, $\{a\}^{\perp_A\perp_A\perp}$ is a common complement of $\{a\}^{\perp_A\perp_A}$ and $\{a\}^{\perp\perp}$ and likewise for $a^*$.  So by \autoref{simper}, if $a$ witnesses $B\sim_AC$ then $B\sim\{a\}^{\perp\perp}\sim\{a^*\}^{\perp\perp}\sim C$ which, by \autoref{simtran}, gives $B\sim C$.
\end{proof}

By \autoref{ess}, the following result applies to any essential ideal $I$ of $S$.

\begin{thm}\label{simsub}
If $I=I^*$ is a bi-ideal with $\{A\cap I:A\in\mathscr{P}(S)^\perp\}\subseteq\mathscr{P}(I)^{\perp_I}$  and $\sim_I$ is reflexive then, for all $B,C\in\mathscr{P}(I)^{\perp_I}$, we have $B\sim_IC\Leftrightarrow B^{\perp\perp}\sim C^{\perp\perp}$.
\end{thm}

\begin{proof}
For any $a\in I$, $\{a\}^{\perp\perp}\cap I\subseteq\{a\}^{\perp_I\perp_I}$.  By assumption $\{a\}^{\perp\perp}\cap I$ is a member of $\mathscr{P}(I)^{\perp_I}$, necessarily containing $a^*a$, and hence $\{a\}^{\perp_I\perp_I}\subseteq\{a\}^{\perp\perp}\cap I$ so
\begin{equation}\label{app}
\{a\}^{\perp\perp}\subseteq\{a\}^{\perp_I\perp_I\perp\perp}=(\{a\}^{\perp\perp}\cap I)^{\perp\perp}\subseteq\{a\}^{\perp\perp\perp\perp}=\{a\}^{\perp\perp}.
\end{equation}
Thus if $a$ witnesses $B\sim_IC$ then $a$ also witnesses $B^{\perp\perp}\sim C^{\perp\perp}$.

Conversely, say $\{a\}^{\perp\perp}=B^{\perp\perp}$ and $\{a^*\}^{\perp\perp}=C^{\perp\perp}$.  As $\sim_I$ is reflexive, we have $b\in B_+$ and $c\in C_+$ with $B=\{b\}^{\perp_I\perp_I}$ and $C=\{c\}^{\perp_I\perp_I}$.  By \eqref{app}, we have $\{b\}^{\perp\perp}=B^{\perp\perp}=\{a\}^{\perp\perp}$, so \autoref{simlem} yields $\{ab\}^{\perp\perp}=B^{\perp\perp}$ and $\{ba^*\}^{\perp\perp}=C^{\perp\perp}$.  Likewise $\{cab\}^{\perp\perp}=B^{\perp\perp}$ and $\{ba^*c\}^{\perp\perp}=C^{\perp\perp}$.  As $I$ is a bi-ideal, $cab\in I$ and $B\sim_IC$ because $\{cab\}^{\perp_I\perp_I}=B^{\perp\perp}\cap I=B$ and $\{ba^*c\}^{\perp_I\perp_I}=C^{\perp\perp}\cap I=C$.
\end{proof}

\section{*-Rings}\label{*R}

Many theorems of the previous section relied on properties of $\sim$ like reflexivity and $\perp$-additivity which do not hold in an arbitrary *-semigroup.  In this section we show that they do at least hold in some fairly general classes of *-rings.

\begin{dfn}
A \emph{*-ring} is simultaneously a ring and a *-semigroup with respect to both $+$ and $\cdot$.  A *-ring is \emph{proper} if it is proper with respect to $\cdot$.
\end{dfn}

\begin{prp}\label{add}
If $S$ is a proper *-ring then
\begin{enumerate}
\item $\sim$ is $\perp$-additive,
\item $S$ is $\perp$-cancellative, and
\item $r\mapsto\{r\}^{\perp\perp}$ is injective on $S_\mathrm{proj}$.
\end{enumerate}
\end{prp}

\begin{proof}\
\begin{enumerate}
\item For any $a,b\in S$, we have $\{a\}^\perp\cap\{b\}^\perp\subseteq\{a+b\}^\perp$, by the distributivity of $\cdot$ over $+$.  Furthermore, if $a^*b=0$ and $(a+b)c=0$ then $0=a^*(a+b)c=a^*ac$ and hence $ac=0$, by properness.  Likewise if $b^*a=0$ and $(a+b)c=0$ then $bc=0$.  As $c$ was arbitrary, $\{a+b\}^\perp\subseteq\{a\}^\perp\cap\{b\}^\perp$ and hence $\{a+b\}^{\perp\perp}=\{a\}^{\perp\perp}\vee\{b\}^{\perp\perp}$.  Thus if $A,B,C,D\in\mathscr{P}(S)^\perp$, $A\perp B$, $C\perp D$, $a$ witnesses $A\sim C$ and $b$ witnesses $B\sim D$ then $a+b$ witnesses $A\vee B\sim C\vee D$.
\item If $\{a\}^\perp\subseteq\{b\}^\perp$ and $as=at$, $a(s-t)=0$ so $(s-t)(s-t)^*\in\{a\}^\perp\subseteq\{b\}^\perp$.  By \autoref{s*st}, $b(s-t)=0$ and hence $bs=bt$.
\item Say $\{p\}^\perp=\{q\}^\perp$, for $p,q\in S_\mathrm{proj}$.  As $ppq=pq$ and we have just shown $S$ is $\perp$-cancellative, $qpq=qq=q$ so $q(p-q)(p-q)q=qpq-qpq-qpq+q=0$.  By properness, $qp=qq=q$.  Likewise $p=pq=(qp)^*=q^*=q$.
\end{enumerate}
\end{proof}

\begin{dfn}
We call $S$ \emph{$\perp$-seperable} if every $\perp$-subset is countable.
\end{dfn}

\begin{dfn}
A \emph{*-algebra} $S$ is simultaneously a *-ring and a real algebra such that $(\lambda a)^*=\lambda a^*$, for all $a\in S$ and $\lambda\in\mathbb{R}$.  A \emph{normed *-algebra} $S$ is a *-algebra together with an algebra norm $||\cdot||$, specifically a function from $S$ to $\mathbb{R}$ such that $||a+b||\leq||a||+||b||$, $||ab||\leq||a||||b||$ and $||\lambda a||\leq|\lambda|||a||$, for all $a,b\in S$ and $\lambda\in\mathbb{R}$.
\end{dfn}

\begin{prp}
Every proper norm separable normed *-algebra $S$ is $\perp$-separable.
\end{prp}

\begin{proof}
Given $a,b\in S$, we have $||a^*a||=||a^*(a-b)||\leq||a^*||||a-b||$.  Noting that $||(\lambda a)^*(\lambda a)||=\lambda^2||a^*a||$ while $||\lambda a^*||=|\lambda|||a^*||$, we see that by replacing $a$ by $\lambda a$, for sufficiently large $\lambda$, we can ensure that $||a^*a||/||a||\geq1$.  Given any $\perp$-subset $A$ of $S$, this means we can replace elements of $A$ by scalar multiples to obtain a subset $A$ of the same cardinality with $||a-b||\geq1$, for all distinct $a,b\in A$.  Thus if $S$ were not $\perp$-separable then it could not be norm separable.
\end{proof}

Recall that a \emph{Banach} *-algebra is a complete normed *-algebra.

\begin{prp}
If $S$ is a proper $\perp$-seperable Banach *-algebra, $\sim$ is $\perp$-complete.
\end{prp}

\begin{proof}
Say $(a_n)\subseteq S$ and $a_ma_n^*=0=a_m^*a_n$ for all distinct $m,n\in\mathbb{N}$.  By replacing each $a_n$ with $\lambda_na_n$, for sufficiently small $\lambda_n>0$, we may assume $\sum||a_n||<\infty$ and hence $a=\sum a_n$ exists.  Multiplication is continuous so, as in the proof of \autoref{add}, $\{a\}^{\perp\perp}=\bigvee\{a_n\}^{\perp\perp}$ and hence $\sim$ is countably $\perp$-additive.  As every $\perp$-subset of $S$ and hence of $\mathscr{P}(S)^\perp$ is countable, $\sim$ is $\perp$-complete.
\end{proof}

Recall that an orthoposet $\mathbb{P}$ is \emph{orthomodular} if $p=q\vee(p\wedge q^\perp)$ whenever $q\leq p$.

\begin{prp}\label{omodref}
If $S$ is a proper $\perp$-seperable Banach *-algebra and $\mathscr{P}(S)^\perp$ is orthomodular then $\sim$ is reflexive.
\end{prp}

\begin{proof}
Given $A\in\mathscr{P}(S)^\perp$, let $B$ be a maximal $\perp$-subset of $A_+$ so $B^{\perp\perp}\subseteq A^{\perp\perp}=A$.  If $B^{\perp\perp}\neq A$ then $A\cap B^\perp\neq\{0\}$, by orthomodularity, so we have $a\in A_+\cap B^\perp$, contradicting the maximality of $B$.  As $S$ is $\perp$-separable, $B$ is countable and, as above, we have $a\in S$ with $\{a\}^{\perp\perp}=B^{\perp\perp}=A$.
\end{proof}

There are norm separable C*-algebras $S$ for which $\mathscr{P}(S)^\perp$ is not orthomodular (see \cite{Bice2014} Example 3.87).  In this case the following result can be applied instead.

Note we are letting $S_{++}$ denote the additive semigroup generated by $S_+$.

\begin{prp}\label{lastprop}
If $S$ is a proper norm separable Banach *-algebra and, for all $a\in S_{++}$, we have $\lambda_a$ with $||a||\leq\lambda_a||a+b||$, for all $b\in S_{++}$, then $\sim$ is reflexive.
\end{prp}

\begin{proof}
Given $A\in\mathscr{P}(S)^\perp$, let $B$ be a countable dense subset of $A_+$.  The continuity of multiplication immediately yields $B^{\perp\perp}=A_+^{\perp\perp}=A^{\perp\perp}$.  By replacing elements of $B$ with non-zero scalar multiplies if necessary, we may assume we have an enumeration $(b_n)$ of $B$ and $\sum||b_n||<\infty$ and hence $a=\lim a_n$ exists, where $a_n=\sum_{m\leq n}b_m$ for all $n\in\mathbb{N}$.  For any $s\in\{a\}^\perp_+$ and $m,n\in\mathbb{N}$ with $m\leq n$, we have $||sb_ms||\leq\lambda_{sb_ms}||sa_ns||\rightarrow0$, as $n\rightarrow\infty$, and hence $b_ms=0$.  Thus $\{a\}^\perp\subseteq B^\perp$ and also $B^\perp\subseteq\{a\}^\perp$, by the continuity of multiplication, so $\{a\}^{\perp\perp}=B^{\perp\perp}=A$.
\end{proof}

When $S$ is a C*-algebra, $S_{++}=S_+$ and we can take $\lambda_a=1$, for all $a\in S_+$, to see that the above result applies.  More generally, \autoref{lastprop} applies whenever $(S,S_{++})$ is $\lambda$-normal (see \cite{AsimowEllis1980} Chapter 2 \S1).

\newpage

\bibliography{maths}{}
\bibliographystyle{alphaurl}

\end{document}